\documentclass[letterpaper, 10 pt, conference]{ieeeconf} 

\IEEEoverridecommandlockouts 
\title{\LARGE \bf
Duality between polyhedral approximation of value functions and optimal quantization of measures
}

\author{Abdellah Bulaich Mehamdi, Wim van Ackooij, Luce Brotcorne, St\'ephane Gaubert and Quentin Jacquet
\thanks{Abdellah Bulaich Mehamdi, Wim van Ackooij and Q. Jacquet are with EDF Lab Paris-Saclay, Palaiseau, France {\tt \{abdellah.bulaich-mehamdi, wim.van-ackooij, quentin.jacquet\}@edf.fr}}%
\thanks{Luce Brotcorne is with INRIA {\tt luce.brotcorne@inria.fr}}%
\thanks{Abdellah Bulaich Mehamdi and St\'ephane Gaubert are with INRIA and CMAP, \'Ecole polytechnique, Institut Polytechnique de Paris, Palaiseau, France {\tt stephane.gaubert@inria.fr}}%
}

\let\labelindent\relax
\usepackage{enumitem}
\usepackage{dsfont}
\usepackage{amssymb}
\usepackage{lipsum}
\usepackage{algorithm}
\usepackage{algpseudocode}
\usepackage{url}
\usepackage{multicol}
\usepackage{hyperref}
\usepackage{amsmath} 
\usepackage{adjustbox}
\usepackage{stmaryrd}
\usepackage{array}
\usepackage{moresize}
\usepackage{float}
\usepackage{bbold}
\usepackage{siunitx}
\usepackage{subfig}
\usepackage{textcomp}
\usepackage{mathtools}
\usepackage{notoccite}
\usepackage[english]{babel}
\usepackage{algpseudocode}
\usepackage{lipsum}
\usepackage{tikz}
\usepackage{circuitikz}
\usetikzlibrary{positioning}
\usepackage{csquotes}
\usepackage{wallpaper}
\usepackage[capitalise]{cleveref}
\usepackage{mathrsfs}

\usepackage{amsthm}
\usepackage{graphicx}
\usepackage{pgfplots}
\usepackage{xcolor}
\usepackage{tikz-3dplot}
\newcommand{\enewton}{\mathscr{P}}
\definecolor{myblue}{RGB}{0, 102, 204} 
\definecolor{myred}{RGB}{204, 0, 0} 
\definecolor{myorange}{RGB}{255, 140, 0} 
\definecolor{mygreen}{RGB}{0, 153, 76} 
\definecolor{myblack}{RGB}{20, 20, 20} 

\newtheorem{theorem}{Theorem}
\newtheorem{corollary}{Corollary}
\newtheorem{proposition}{Proposition}
\newtheorem{definition}{Definition}

\theoremstyle{definition}

\newtheorem{remark}{Remark}

\usepackage{todonotes}

\newcommand{\val}{\operatorname{val}}

\newcommand{\argmax}{\operatornamewithlimits{argmax}}

\newcommand{\MA}{\operatorname{MA}}

\newcommand{\R}{\mathbb{R}}
\newcommand{\x}{\mathbf{x}}

\begin{document}

\maketitle

\begin{abstract}
Approximating a convex function by a polyhedral function that has a limited number of facets is a fundamental problem with applications in various fields, from mitigating the curse of dimensionality in optimal control to bi-level optimization. We establish a connection between this problem and the optimal quantization of a positive measure. Building on recent stability results in optimal transport, by Delalande and M\'erigot, we deduce that the polyhedral approximation of a convex function is equivalent to the quantization of the Monge-Amp\`ere measure of its Legendre-Fenchel dual. This duality motivates a simple greedy method for computing a parsimonious approximation of a polyhedral convex function, by clustering the vertices of a Newton polytope. We evaluate our algorithm on two applications: 1) A high-dimensional optimal control problem (quantum gate synthesis), leveraging McEneaney's max-plus-based curse-of-dimensionality attenuation method; 2) A bi-level optimization problem in electricity pricing. Numerical results demonstrate the efficiency of this approach.
\end{abstract}

\section{Introduction}
\subsection{Motivation}\label{subsec-motiv}
Polyhedral approximation is a fundamental technique in computational geometry and optimization, particularly in high-dimension. By approximating complex convex bodies with simpler polytopes or functions, one can reduce computational complexity while preserving an accurate representation of the original set. This approach is widely used in various mathematical and applied fields, including optimal control, program verification, and economic modeling.

Notably, McEneaney introduced a max-plus method to approximate the value function of an optimal control problem~\cite{McEneaney2006}. His approach represents the value function as a supremum or ``max-plus linear combination'' of elementary basis functions. In particular, affine basis functions lead to a polyhedral approximation of the value function. Several max-plus type methods have been developed~\cite{flemingmceneaney,McEneaney_2007,Akian_2008}; the methods developed following~\cite{McEneaney_2007} have the advantage to attenuate the curse of dimensionality. An essential ingredient here is the {\em pruning} of the polyhedral representation, that is, given a fine polyhedral approximation, the construction of a new - reduced complexity -approximation, minimizing the ``pruning error''~\cite{qu-gaubert-2011}. A similar pruning problem arises when implementing the SDDP approach~\cite{Pereira}, in which polyhedral approximations of the value function are also constructed.

A related pruning problem appears in bi-level optimization, particularly in adverse selection models studied by Rochet and Chon\'e~\cite{rochet-chone-1998}. In optimal nonlinear pricing, firms must optimize a function under incentive compatibility constraints, leading to convexity constraints and a setting in which polyhedral approximation plays a crucial role in solving the problem efficiently by discretization~\cite{carlier-lrobert-maury-2000, ekeland-mbromberg-2009, bergemann-yeh-zhang-2021}. 

\subsection{Contribution}
We relate the optimal polyhedral approximation of a convex function and the optimal quantization of a probability measure, exploiting a duality between both problems. More precisely, we show that the approximation error of a convex function by a polyhedral function with a prescribed budget (number of facets) can be controlled by the quantization error of the Monge-Amp\`ere measure associated with its Legendre–Fenchel dual, and vice versa; see Theorem~\ref{theorem:principal}. In particular, this result highlights how tools from measure theory provide a framework for understanding the polyhedral approximation problem. Corollary~\ref{corollary:principal} further clarifies the connection and describes how an $\epsilon$-approximate solution to one problem can be derived from the solution to the other. This builds on recent stability results in optimal transport, by Delalande an M\'erigot~\cite{delalande:tel-03935445, delalande2023quantitative}. Like optimal polyhedral approximation, optimal quantization is subject to a curse of dimensionality. Nevertheless, we exploit the duality between these two problems to inspire a new pruning approach (Algorithm~\ref{algorithm:dqh}). Instead of solving exactly the quantization problem, we perform a clustering of the vertices of the lifted Newton polytope of the Legendre–Fenchel transform of a given polyhedral function. We solve this clustering problem by a $k$-center greedy algorithm which provides an optimal solution up to a factor $2$~\cite{Gonzalez1985ClusteringTM}. Proposition~\ref{proposition:newton} shows that the final approximation error is controlled in terms of the clustering error. 

Then, we assess the proposed pruning method by conducting experiments on an optimal control problem and a nonlinear pricing problem. For the control problem, we consider quantum gate synthesis for a 2-qubit problem, following the original approach introduced by Sridharan et al.~\cite{Sridharan_2010,Sridharan2014}. The authors developed a max-plus based method attenuating the curse of dimensionality for this challenging optimal control problem, where the challenge stems from its high dimension (15) and subriemannian nature. This method propagates a sequence of piecewise linear functions defined over the group of special-unitary matrices. However, the ``pruning operation'' is the bottleneck of the method, as it relies on an importance metric that requires a {\em semi-definite programming (SDP)} to be solved, see~\cite{Gaubert_2014}. For the nonlinear pricing problem, we consider an application in electricity market where the tariffs are derived from the discretized linear-quadratic Rochet-Chon\'e model~\cite{rochet-chone-1998,carlier-lrobert-maury-2000,bergemann-yeh-zhang-2021}. This results in a very large set of offers, which is difficult to represent efficiently to the clients. We applied our algorithm to select a subset of the tariffs (the menu), giving a more efficient representation while minimizing revenue deterioration (see~\cite{jacquet-ackooij-2023}).

In both applications, the new pruning algorithm demonstrates strong performance in terms of both computational time and solution quality. In particular, by comparison with earlier pruning methods based on importance metrics~\cite{Sridharan_2010, Sridharan2014, Gaubert_2014,jacquet-ackooij-2023}, we get an improved solution.

\subsection{Related work}
As discussed in \Cref{subsec-motiv}, a strong motivation to study the best polyhedral approximation problem arises from the attenuation of curse of dimensionality in optimal control, and especially from max-plus type methods, developed after Fleming and McEneaney~\cite{flemingmceneaney}, by several authors, see~\cite{Akian_2008,Sridharan_2010,Sridharan2014, Gaubert_2014,mceneaneydower}. Similar polyhedral approximation problems arise in the implementation of stochastic dual dynamic programming, see~\cite{Pereira, Philpott, leclere}. Recently, the polyhedral approximation problem has been studied with motivations of mathematical economy and bi-level programming, in the setting of the Rochet-Chon\'e model~\cite{bergemann-yeh-zhang-2021, jacquet-ackooij-2023}.

Beyond applications in control theory and economics, polyhedral approximation issues arise in convex geometry. Then, one is interested in approximating a convex body by a polytope with a prescribed number $n$ of facets. For strictly convex bodies with a $\mathcal{C}^2$ boundary, the approximation error decreases as $O(\frac{1}{n^{2/(d-1)}}) \text{ as } n \to \infty$, where $d$ is the dimension of the ambient space (see~\cite[Chapter 11]{gruber2007convex}, and~\cite{bronstein-2008}). The constant involved in the big-O involves a ``Gaussian curvature'', showing that the approximation is easier in ``flat'' regions of the boundary of the convex body.

Previous numerical studies have focused on {\em pruning} polyhedral approximations—that is, refining polyhedral representations of functions—to improve representational efficiency. Gaubert, McEneaney and Qu showed that the pruning problem can be interpreted as a continuous-space facility location problem, using a specific Bregman-type distance~\cite{qu-gaubert-2011}. They also adapted the approximation results known for convex bodies to the case of functions, showing that the approximation error is of order $O(\frac{1}{n^{2/d}})$ for smooth and strictly convex functions, with an implied constant also of Gaussian curvature type.

This interpretation in terms of facility location leads to the minimization of a supermodular function—an optimization problem that has been extensively studied in the literature~\cite{krause2014submodular, boutsidis2015greedyminimizationweaklysupermodular}. In practice, greedy heuristics are widely adopted due to their simplicity and scalability. Several specific pruning algorithms were developed~\cite{4587234,qu-gaubert-2011,jacquet-ackooij-2023}, building on notions of ``importance metric'', quantifying the notion of most useful function among a given family. Then, these metrics are exploited by heuristics, removing the less critical functions until a desired approximation budget or level of fidelity is achieved. Interestingly, importance metrics are (indirectly) related to the volume of the Laguerre cells arising in the optimal quantization problem, hence our results provide further grounds for the use of such metrics.

We finally note that the same order of error, $O(\frac{1}{n^{2/d}})$, appears in polyhedral approximation and in optimal quantization, see~\cite{Pags2004, merigot2021non}.

The paper is organized as follow. In Section~\ref{section:preliminaries}, we introduce the main tools and notation. In Section~\ref{section:principal}, we connect polyhedral approximation and measure quantization through the Monge-Amp\`ere measure and optimal transport, establishing their relationship. We then present our method for computing an inner approximations of a polyhedral function. Finally, in Section~\ref{section:applications}, we illustrate the method by a quantum gate synthesis problem and a nonlinear pricing problem originating from the electricity market.

\subsection{Acknowledgment}
The authors thank Quentin M\'erigot for very helpful discussions on stability results in optimal transport, which contributed to the development of this work.

\section{Preliminaries} \label{section:preliminaries}
We next introduce the notation and main tools.
\subsection{Notation}
For an integer number $n \geq 0$, we introduce the following notation:
\begin{itemize}
\item the unit simplex $\Delta_n = \{\lambda \in \R^n_+ \mid \sum_{k = 1}^n \lambda_k = 1 \}$.
\item the set $[n]$ is $\{1, \cdots, n\}$.
\item the set $P(K)$ of probability measures on a compact metric space $K$.
\item a compact convex set $Q \subset \R^d$.
\item a quadratic cost $c: (x,y) \mapsto \|x-y\|^2$.
\item the dirac measure $\delta_x$, assigning the mass $1$ to $x$ and $0$ elsewhere.
\item the operator $\val$ maps optimization problems to their optimal values.
\end{itemize}
Additionally, $\rho$ is a probability measure supported on a compact convex set $X \subset \R^d$, absolutely continuous with respect to the Lebesgue measure, with a density bounded from above and below by positive constants.

\subsection{Optimal transport}
Optimal transport is a mathematical framework for efficiently moving mass between two distributions while minimizing a given cost function. Originally introduced by Monge in 1781, it seeks to transport mass from a source distribution $\mu$ to a target distribution $\nu$ while minimizing a given cost function. The Monge formulation requires finding a transport map $T: \mathbb{R}^d \to \mathbb{R}^d$ that pushes $\mu$ onto $\nu$, i.e., $T_{\#} \mu = \nu$, and minimizes the transport cost:
\begin{equation*}
\min_{T} \int_{\mathbb{R}^d} c(x, T(x)) \, d\mu(x),
\end{equation*}
where $c(x, y)$ represents the cost of moving mass from $x$ to $y$. Monge’s original problem requires a deterministic transport, which may lead to non-existence of solutions unless additional regularity conditions are met. To overcome the non-existence problem, Kantorovich introduced a relaxed formulation where mass can be split. Instead of seeking a transport map, one looks for a transport plan $\gamma \in \Gamma(\mu, \nu)$, i.e., a probability measure on $\mathbb{R}^d \times \mathbb{R}^d$ with marginals $\mu$ and $\nu$. The Kantorovich problem is then:
\begin{equation*}
K(\mu,\nu) := \min_{\gamma \in \Gamma(\mu, \nu)} \int_{\mathbb{R}^d \times \mathbb{R}^d} c(x, y) \, d\gamma(x, y).
\end{equation*}
\begin{remark}
When $c$ is $c(x, y) = {\|x - y\|}_p^p$, $W_p(\mu,\nu) := K(\mu,\nu)^{\frac{1}{p}}$, defines a distance between $\mu$ and $\nu$, known as the Wasserstein distance of order $p$, with $p \in [1, \infty]$.
\end{remark}
\begin{definition}[{\cite{Villani2009}, Page 11}] \label{definition:push-forward}
If $\mu$ is a Borel measure on $\R^d$, and $T$ is a Borel map $\R^d \to \R^d$, then $T_\#\mu$ stands for the image measure (or \emph{push-forward}) of $\mu$ by $T$. It is a Borel measure on $\R^d$, defined by $T_\#\mu(A) = \mu(T^{-1}(A))$. 
\end{definition}
A fundamental result in optimal transport is Brenier’s theorem~\cite{brenier1991polar}, which provides a strong link between Monge’s and Kantorovich’s formulations when the cost function is the squared Euclidean distance, i.e., $c(x, y) = \|x - y\|^2$. Under mild regularity assumptions (e.g., $\mu$ is absolutely continuous with respect to the Lebesgue measure), there exists a unique optimal transport map $T$, which is the gradient of a convex function $\varphi$ such that $T = \varphi$ almost everywhere. The function $\varphi$ is called the \textit{Brenier potential} and solves the Monge-Amp\`ere equation.
\begin{theorem}[\cite{delalande:tel-03935445}, Theorem 5.12, and Remark 5.13] \label{theorem:delalande}
For any probability measures $\mu, \nu$ on $Q$, there exist two constants $C_L, C_R$ depending only on $\rho, d, X,$ and $Q$ such that:
\begin{align*}
C_L W_2(\mu, \nu)^3 \leq \| u - v \|_{L_2(\rho)} \leq C_R W_1(\mu, \nu)^\frac{1}{2}, 
\end{align*}
where $u$ and $v$ are the Brenier potentials associated with $\rho$ and the measures $\mu, \nu$, respectively.
\end{theorem}
\subsection{Monge-Amp\`ere measure}
\begin{definition}
The $\rho$-Monge-Amp\`ere measure $\MA_\rho(u)$ is defined by:
\begin{equation*} 
\MA(u)(E) := \rho(\partial u(E)), \quad \forall \text{ Borel set } E \subseteq \R^d,
\end{equation*}
where $\partial u(E) = \bigcup_{x \in E} \partial u(x)$. 
\end{definition}
\begin{remark}
When $\rho$ is the Lebesgue measure, and $u \in C^2(\R^d)$, the change of variable formula gives:
\begin{align*}
\rho(\partial u(E)) = \rho(\nabla u(E)) = \int_E \text{det} \, D^2 u(x) \, \rho(x) \, dx,
\end{align*}
for every Borel set $E \subseteq \R^d$, therefore $\MA_\rho(u) = \text{det} \, D^2 u(x) \, \rho(x) \, dx$~\cite{dephilippis2013mongeampereequationlinkoptimal}. 
\end{remark} 
\begin{definition} \label{definition:fenchel}
Consider a convex function $u : \R^d \rightarrow \R$, the convex conjugate or Legendre-Fenchel transform of $u$ is the function $u^*: \R^d \to \mathbb{R}$ defined by: 
\begin{equation*}
u^*(y) = \sup_{x \in \R^d} \big(y \cdot x - u(x)\big).
\end{equation*}
\end{definition}
\begin{remark} \label{remark:principal}
For any proper closed convex function $u$~\cite[Chapter 2]{vanAckooij_Oliveira_2025}, Theorem 23.5 in~\cite{rockafellar-1970} ensures that $\partial u^* = \partial u^{-1}$. This leads to $MA_\rho(u^*)(E) = \rho(\partial u^*(E)) = \rho(\partial u^{-1}(E))$, and by Definition~\ref{definition:push-forward}, we have that: $(\nabla u)_\# \rho = MA_\rho(u^*)$.
\end{remark}
Theorem~\ref{theorem:gu} provides a polyhedral solution $u_n$ to the Monge-Amp\`ere problem $\MA_\rho(u^*) = \sum_{i=1}^m \nu_k \, \delta_{q_k}$.
\begin{theorem}[\cite{gu2013variationalprinciplesminkowskitype}, Theorem 1.2] \label{theorem:gu}
Let $\{q_1, \cdots, q_n\}$ be a set of distinct points in $\R^d$. Then for any $\nu \in \Delta_n$, there exists $p=(p_1, \cdots, p_n) \in \R^n$, unique up to adding a constant $(c,..., c)$, such that:
\begin{equation*}
\int_{ \partial u_n^*(q_k) \cap X}\rho(x) \, dx = \nu_k, \forall k,
\end{equation*}
with $u_n: x \mapsto \max_{k \in [n]} \langle q_k, x \rangle - p_k$. Furthermore, $\nabla u_n$ minimizes the quadratic cost $\int_{X} |x - T(x)|^2 \, \rho \, dx$ among all transport maps $T: (X, \rho \, dx) \to (\R^d$, $\sum_{i=1}^n \nu_k \, \delta_{q_k})$. 
\end{theorem}

\section{Duality between Polyhedral Approximation and Quantization} \label{section:principal}
In this section, we describe a duality between the optimal polyhedral approximation of a convex function with respect to the distance $L_2(\rho)$ and the optimal quantization of a probability measure with respect to the Wasserstein distance. Then, inspired by this duality, we propose a practical method that combines {\em linear programming (LP)} or {\em semi-definite programming (SDP)} (depending on the situation at hand) with clustering of the vertices of a lifted Newton polytope to produce scalable polyhedral approximations.

\subsection{Connection between polyhedral approximation and optimal transport}
Given a set $Q$ and a budget $n$, the optimal polyhedral approximation problem for a proper closed convex function $u$ with a prescribed budget and subject to subgradient constraints, with respect to the distance $L_2(\rho)$, can be formulated as $\mathcal{P}_n(u)$:
\begin{align} \label{problem:polyhedral} 
\min_{(q_k, p_k)_{k=1}^n \in Q \times \mathbb{R}} &{\left(\int_X (u(x) - u_n(x))^2 \, \rho(x) \, dx\right)}^\frac{1}{2} \\
\text{s.t.:} &\int_{X} (u(x) - u_n(x)) \, \rho(x) \, dx = 0, \notag
\end{align}
where $u_n: x \mapsto \max_{k \in [n]} \langle q_k, x \rangle - p_k$. This can be compared with the optimal quantization of a probability measure $\mu$ on $Q$ with a prescribed budget $n$ with respect to the Wasserstein-$p$ distance, $\mathcal{Q}_n^p(\mu)$:
\begin{align} \label{problem:quantization}
\min_{\nu \in P(Q), |\text{supp}(\nu)| \leq n} W_p(\mu, \nu).
\end{align}
Theorem~\ref{theorem:principal} relates the optimal polyhedral approximation of a convex function $u$ and the optimal quantization of the Monge-Amp\`ere measure associated with its Legendre–Fenchel transform, and provides the connection between their solutions. Figure~\ref{fig:framework} visually summarizes this duality.
\begin{theorem} \label{theorem:principal}
Let $u: \R^d \rightarrow \R$ be a proper closed convex function, and $\mu = \MA_\rho(u^*)$. Then the problems~\eqref{problem:polyhedral} and~\eqref{problem:quantization} are equivalent, i.e., there exists constants $C_L, C_R>0$ that depend only on $\rho$, $d$, $X$, and $Q$ such that:
\begin{equation*}
C_L (\val\mathcal{Q}_n^2(\mu))^3 \leq \val(\mathcal{P}_n(u)) \leq C_R (\val\mathcal{Q}_n^1(\mu))^\frac{1}{2}. \label{eq:key}
\end{equation*}
\end{theorem}
\begin{proof} 
Let $u: \R^d \rightarrow \R$ be a proper closed convex function, and let $u_n$ be a solution to problem~\eqref{problem:polyhedral}. We have by Remark~\ref{remark:principal} that $\nabla u$ and $\nabla u_n$ are the Brenier potentials associated with $\rho$ and the measures $\mu = \MA_\rho(u^*), \MA_\rho(u_n^*)$, respectively. By Theorem~\ref{theorem:delalande}, we establish: 
\begin{equation*}
C_L W_2(\mu, \MA_\rho(u_n^*))^3 \leq \| u - u_n \|_{L_2(\rho)} := \val(\mathcal{P}_n(u)). 
\end{equation*} 
Since $\MA_\rho(u_n^*)$ is none other than the sum of Dirac functions concentrated on the slopes $(q_k)_{k \in [n]}$ of the function $u_n$ and weighted by $(\rho(\partial u_n^*(q_k)))_{k \in [n]}$, we obtain a feasible point and consequently: $C_L (\val\mathcal{Q}_n^2(\mu))^3 \leq \val(\mathcal{P}_n(u))$. On the other hand, let $\mu = \MA_\rho(u^*)$, and let $\nu = \sum_{k = 1}^{n} \nu_k \, \delta_{q_k}$ be a solution to the problem $\mathcal{Q}_n^1(\mu)$. By Theorem~\ref{theorem:gu}, there exists a polyhderal function $u_n: x \mapsto \max_{k \in [n]} \langle q_k, x \rangle - p_k$ such that $MA_\rho(u_n^*) = \nu$ unique up to adding a constant $(c,..., c)$. The constant $c$ can be chosen so that $\int_X u_n(x) \, \rho(x) \, dx = \int_X u(x) \, \rho(x) \, dx$. This yields a feasible solution to problem~\eqref{problem:polyhedral} and using Theorem~\ref{theorem:delalande} once again, we obtain the other inequality from Theorem~\ref{theorem:principal}.
\end{proof}
\begin{figure}[h!]
\vspace{-0.35cm} 
  \centering
  \begin{tikzpicture}[
    node distance=1.6cm,
    >=Stealth,
    every node/.style={align=center}
  ]

  \node (12) {};
  \node[right=3.6cm of 12] (22) {};
  \node[below=of 22] (21) {};
  \node[below=of 12] (11) {};

  \node[left=0cm of 12] {$u$};
  \node[right=0cm of 22] {$\mu = \MA_{\rho}(u^*)$}; 
  \node[right=0cm of 21] {$\mu_n$};
  \node[left=0cm of 11] {$u_n$}; 

  \draw[->] (12) -- (22) node[midway,above] {Monge-Ampère measure};
  \draw[->] (22) -- (21) node[midway,right] {Measure\\ quantization};
  \draw[->] (21) -- (11) node[midway,below] {Monge-Ampère equation};
  \draw[->] (12) -- (11) node[midway,left] {Polyhedral\\ approximation};

  \end{tikzpicture}
  \caption{Relationship between polyhedral approximation and measure quantization (Theorem~\ref{theorem:principal}).}
  \label{fig:framework}
\vspace{-0.35cm} 
\end{figure}
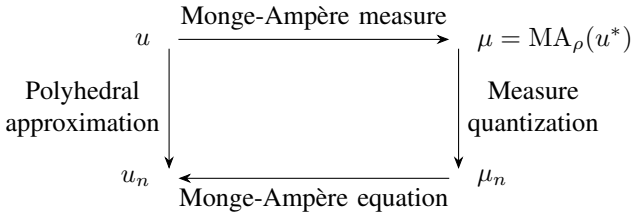
Using the notation $\epsilon_\mathcal{Q} = C_L^{-\frac{1}{3}}$ and $\epsilon_\mathcal{P} = C_R$, we derive Corollary~\ref{corollary:principal} from Theorem~\ref{theorem:principal} by applying exploiting the previous inequalities.
\begin{corollary} \label{corollary:principal}
  Let $p \in [1,2]$, $u: \R^d \rightarrow \R$ be a proper closed convex function, and $\mu = \MA_\rho(u^*)$. Then if $v$ is the optimal solution of problem~\eqref{problem:polyhedral}, then $\MA_\rho(v^*)$ is an $\epsilon_\mathcal{Q} \, {val(\mathcal{P}_n(\mu))}^{\frac{1}{3}}$-optimal solution of~\eqref{problem:quantization}. Conversely, if $\nu$ is the optimal solution of the latter problem, then the convex solutions $v$ to the $\rho$-Monge-Amp\`ere problem $\MA_\rho(v^*) = \nu$ are $\epsilon_\mathcal{P}\, {val(\mathcal{Q}_n^p(u))}^{\frac{1}{2}}$-optimal for~\eqref{problem:polyhedral}.
\end{corollary}
Corollary~\ref{corollary:principal} asserts that if the error is well controlled for one of the two problems, then the other is also well controlled and yields a ``near"-optimal solution. 
\begin{remark}
The error estimates on polyhedral approximation in~\cite{gruber2007convex} and on optimal quantization in~\cite{Pags2004} cannot be deduced from one another via Theorem~\ref{theorem:delalande}.
\end{remark}

\if{
\begin{remark}
In the one-dimensional case with $\rho$ the Lebesgue measure on $X=[0,1]$, the solution to the Monge-Amp\`ere problem $\MA_\rho(v^*) = \sum_{k=1}^m \nu_k \, \delta_{q_k}$ can be simply computed by integration of Heaviside functions. In fact, up to an additive constant, the function:
\begin{equation*}
v(x) = \sum_{k=1}^m (q_k - q_{k-1}) \, \sum_{l = 0}^{k-1} \max(0, x - \nu_l),
\end{equation*}
is a solution, where $q_0 = 0$ and $(q_k)_{k \in [m]}$ are ordered.
\end{remark}
}\fi

\subsection{Polyhedral approximation algorithm: clustering the vertices of the lifted Newton polytope} \label{subsection:geometric}
To obtain a polyhedral approximation of a convex function, \Cref{theorem:principal} suggests quantizing the Monge-Amp\`ere measure of its Legendre-Fenchel dual, and then recovering the polyhedral approximation by solving a semi-discrete Monge-Amp\`ere problem~\cite{sdot, meyron2019initialization}, as per~\Cref{theorem:gu}. This approach is computationally efficient in low dimensions. In the following applications, however, we are interested in high-dimensional examples, but the problem is simplified since our input is already given by a polyhedral function
\begin{align}
  u_N: x \in \R^d \mapsto \max_{k \in [N]} \langle q_k, x \rangle - p_k. \label{e-input}
\end{align}
Our goal is to construct a reduced complexity approximation $u_S: x\mapsto \max_{k \in S} ( \langle q_k, x \rangle - p_k)$ where $S\subset[N]$ must be of a fixed cardinality $n$. This is the {\em pruning problem} studied in~\cite{Sridharan_2010,qu-gaubert-2011}. The previously established connection suggests a ``dual'' heuristic that
operates on the Legendre-Fenchel transform of $u_N$ instead of the original function.
\begin{proposition}[{\cite{hormander2007notions}, Theorem 2.2.7}] 
The convex conjugate of the function $u_N$ is given by:
\begin{equation*}
u^*_N(q) = \inf_{\lambda \in \Delta_N} \sum_{k = 1}^N \lambda_k p_k \text{ s.t. } \sum_{k = 1}^N \lambda_k \, q_k = q.
\end{equation*}
\end{proposition}
Concretely, Legendre-Fenchel conjugacy associates with $u_N$ the function $u_N^*$, whose graph is the lower boundary of the \emph{lifted Newton polytope $\enewton$} in dimension $d+1$, i.e., $\enewton$ is the convex hull of the $N$ vertical rays $(q_k, p_k+\R_{\geq 0})_{k=1}^N$. The non-differentiability locus of the function $u_N^*$ is known as a regular subdivision of the configurations of points $(q_k)_{k=1}^N$, see~\cite[2.2.3]{DeLoera2010} for background.
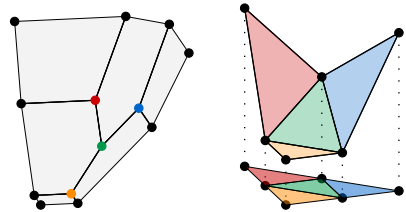
\begin{figure}[h!]
\vspace{-0.35cm} 
\centering
\subfloat{
\tdplotsetmaincoords{75}{-10}
\begin{tikzpicture}[tdplot_main_coords,scale=0.5]

\draw [line width=0.5pt,color=black] (-1,0,-1) -- (0,0,-1);
\draw [line width=0.5pt,color=black] (0,-1,-1) -- (0,0,-1);
\draw [line width=0.5pt,color=black] (1,1,0) -- (0,0,-1);
\draw [line width=0.5pt,color=black] (2,1,1) -- (1,1,0);
\draw [line width=0.5pt,color=black] (2,1,1) -- (3,2,3);
\draw [line width=0.5pt,color=black] (2,1,1) -- (2,-1,1);
\draw [line width=0.5pt,color=black] (1,2,1) -- (1,1,0);
\draw [line width=0.5pt,color=black] (1,2,1) -- (2,3,3);
\draw [line width=0.5pt,color=black] (1,2,1) -- (-1,2,1);

\draw[line width=0.2pt,fill=myblack,fill opacity=0.05] (-1,0,-1) -- (0,0,-1) -- (0,-1,-1) -- (-1,-1,-1) -- cycle;
\draw[line width=0.2pt,fill=myblack,fill opacity=0.05] (1,2,1) -- (1,1,0) -- (0,0,-1) -- (-1,0,-1) -- (-1,2,1) -- cycle;
\draw[line width=0.2pt,fill=myblack,fill opacity=0.05] (2,1,1) -- (1,1,0) -- (0,0,-1) -- (0,-1,-1) -- (2,-1,1) -- cycle;
\draw[line width=0.2pt,fill=myblack,fill opacity=0.05] (1,1,0) -- (2,1,1) -- (3,2,3) -- (2,3,3) -- (1,2,1) -- cycle;
\draw[line width=0.2pt,fill=myblack,fill opacity=0.05] (2,1,1) -- (2,-1,1) -- (3, -1, 3) -- (3,2,3) -- cycle;
\draw[line width=0.2pt,fill=myblack,fill opacity=0.05] (1,2,1) -- (-1,2,1) -- (-1, 3, 3) -- (2,3,3) -- cycle;

\node at (3, -1, 3) [circ, color=black, scale=1.] {};
\node at (-1, 3, 3) [circ, color=black, scale=1.] {};
\node at (-1,2,1) [circ, color=black, scale=1.] {};
\node at (2,-1,1) [circ, color=black, scale=1.] {};
\node at (-1,0,-1) [circ, color=black, scale=1.] {};
\node at (0,-1,-1) [circ, color=black, scale=1.] {};
\node at (-1,-1,-1) [circ, color=black, scale=1.] {};
\node at (2,3,3) [circ, color=black, scale=1.] {};
\node at (3,2,3) [circ, color=black, scale=1.] {};
\node at (0,0,-1) [circ, color=myorange, scale=1.] {};
\node at (1,1,0) [circ, color=mygreen, scale=1.] {};
\node at (1,2,1) [circ, color=myred, scale=1.] {};
\node at (2,1,1) [circ, color=myblue, scale=1.] {};
\end{tikzpicture}
}
\subfloat{
\tdplotsetmaincoords{70}{-20}
\begin{tikzpicture}[tdplot_main_coords,scale=0.8]

\draw [line width=0.5pt,color=black] (0,0,1) -- (0,1,1);
\draw [line width=0.5pt,color=black] (0,0,1) -- (1,0,1);
\draw [line width=0.5pt,color=black] (1,0,1) -- (0,1,1);
\draw [line width=0.5pt,color=black] (1,0,1) -- (1,1,2);
\draw [line width=0.5pt,color=black] (0,1,1) -- (1,1,2);
\draw [line width=0.5pt,color=black] (1,0,1) -- (2,0,3);
\draw [line width=0.5pt,color=black] (0,1,1) -- (0,2,3);
\draw [line width=0.5pt,color=black] (2,0,3) -- (1,1,2);
\draw [line width=0.5pt,color=black] (0,2,3) -- (1,1,2);

\draw[line width=0.2pt,fill=myorange,fill opacity=0.3] (0,0,1) -- (0,1,1) -- (1,0,1) -- cycle;
\draw[line width=0.2pt,fill=myblue,fill opacity=0.3] (1,0,1) -- (1,1,2) -- (2,0,3) -- cycle;
\draw[line width=0.2pt,fill=myred,fill opacity=0.3] (1,1,2) -- (0,1,1) -- (0,2,3) -- cycle;
\draw[line width=0.2pt,fill=mygreen,fill opacity=0.3] (1,1,2) -- (0,1,1) -- (1,0,1) -- cycle;

\node at (0,0,1) [circ, color=black, scale=1.] {};
\node at (1,0,1) [circ, color=black, scale=1.] {};
\node at (0,1,1) [circ, color=black, scale=1.] {};
\node at (1,1,2) [circ, color=black, scale=1.] {};
\node at (2,0,3) [circ, color=black, scale=1.] {};
\node at (0,2,3) [circ, color=black, scale=1.] {};

\draw [line width=0.5pt,color=black,loosely dotted] (0,0,1) -- (0,0,0.2);
\draw [line width=0.5pt,color=black,loosely dotted] (1,0,1) -- (1,0,0.2);
\draw [line width=0.5pt,color=black,loosely dotted] (0,1,1) -- (0,1,0.2);
\draw [line width=0.5pt,color=black,loosely dotted] (1,1,2) -- (1,1,0.2);
\draw [line width=0.5pt,color=black,loosely dotted] (2,0,3) -- (2,0,0.2);
\draw [line width=0.5pt,color=black,loosely dotted] (0,2,3) -- (0,2,0.2);

\node at (0,0,0.2) [circ, color=black, scale=1.] {};
\node at (1,0,0.2) [circ, color=black, scale=1.] {};
\node at (0,1,0.2) [circ, color=black, scale=1.] {};
\node at (1,1,0.2) [circ, color=black, scale=1.] {};
\node at (2,0,0.2) [circ, color=black, scale=1.] {};
\node at (0,2,0.2) [circ, color=black, scale=1.] {};

\draw[line width=0.2pt,fill=myorange,fill opacity=0.5] (0,0,0.2) -- (0,1,0.2) -- (1,0,0.2) -- cycle;
\draw[line width=0.2pt,fill=myblue,fill opacity=0.5] (1,0,0.2) -- (1,1,0.2) -- (2,0,0.2) -- cycle;
\draw[line width=0.2pt,fill=myred,fill opacity=0.5] (1,1,0.2) -- (0,1,0.2) -- (0,2,0.2) -- cycle;
\draw[line width=0.2pt,fill=mygreen,fill opacity=0.5] (1,1,0.2) -- (0,1,0.2) -- (1,0,0.2) -- cycle;

\end{tikzpicture}
}
\caption{The graph of the function $u: (x, y) \mapsto \max(-1, x - 1, y - 1, x + y - 2, 2 \, x - 3, 2 \, y - 3)$ (left); its Newton polytope (bottom right); its (three-dimensional) lifted Newton polytope (upper right).}
\label{fig:newton}
\vspace{-0.35cm} 
\end{figure}

Figure~\ref{fig:newton} shows a function $u$ on the left, alongside with its Legendre-Fenchel dual on the right, identified to the lifted Newton polytope $\enewton$. Each black dot in $\enewton$ corresponds to a affine term in~\eqref{e-input}, which defines a linearity region of the primal function (Laguerre cell, left). Pruning a polyhedral function involves selecting a subset of affine functions, to create a polyhedral approximation that closely fits the original function. By duality, this process corresponds to identifying the points  $(q_k, p_k)_{k=1}^N$ that are the ``most essential'' in $\enewton$. To do so, we will use a ``multiple-center'' model. Specifically, we fix a norm on $\R^{d+1}$, and look for a subset of ``centers'' $\{(q_k,p_k) \}_{k\in S}$ with $S \subseteq [N]$, $|S|=n$, solution of the following problem:
\begin{equation} \label{equation:newton} 
\mathcal{E}(n, N) = \min_{S \subseteq [N], |S| = n} \max_{k \in [N]} \min_{l \in S} \|(q_k, p_k) - (q_l, p_l)\|.
\end{equation}
Equivalently, we want to minimize the maximum distance of the points $\{(q_k,p_k)\}_{k\in [N]}$ to a set of ``centers'' of cardinality $n$, that must be chosen. We will use the standard term of ``{\em $k$-center}''~\cite{Gonzalez1985ClusteringTM}, for this problem,although in our setting the number of centers is denoted by $n$ instead of $k$.

The next proposition shows that the approximation error for $u_N$ is controlled by the error of the $k$-center problem.
\begin{proposition} \label{proposition:newton}
Let $S$ be an optimal solution of \eqref{equation:newton}. Then, for all $x \in X$, we have that:
\begin{align*} 
0 \leq u_N(x) - u_S(x) \leq \mathcal{E}(n, N) \| (x, 1) \|_*,
\end{align*}
where $\| \cdot \|_*$ denotes the norm dual to $\| \cdot \|$.
\end{proposition}
\begin{proof}
We have:
\begin{align*}
u_N(x) - u_S(x) & = \max_{k \in [N]} (\langle q_k , x \rangle - p_k) - \max_{l \in S} (\langle q_l , x \rangle - p_l) \\
&= \max_{k \in [N]} \min_{l \in S} \langle(q_k, p_k) - (q_l, p_l), (x, 1) \rangle \\
& \leq \max_{k \in [N]} \min_{l \in S} \|(q_k, p_k) - (q_l, p_l)\| \| (x, 1) \|_* \\
&=  \mathcal{E}(n, N) \| (x, 1) \|_*.
\end{align*}
The non-negativity follows from $S \subseteq [N]$.
\end{proof}
A $k$-center greedy algorithm which provides a $2$-factor approximation with a time complexity of $O(N \, n \, (d+1))$ was introduced in~\cite{Gonzalez1985ClusteringTM}. This algorithm enables us to approximate the lifted Newton polytope by selecting a set of centers $\{(q_k,p_k)\}_{k\in S}$ among its vertices, corresponding to a lower approximation $u_S \leq u_N$. Before applying the $k$-center algorithm, we first ensure that each affine function $x\mapsto \langle q_k, x \rangle - p_k$ actually contributes to the original function $u_N$ on the domain $D$. This guarantees that the vertices of the lifted Newton polytope $\enewton$ correspond precisely to the slopes and intercepts of $u_N$. For each $k \in S$, we solve the following problem:
\begin{align} \tag{$F_{k, S}^D$} \label{equation:feasibility}
\max_{v, x \in D} \; v \\
\text{s.t.} \quad & v \leq \langle q_k - q_l, x \rangle - (p_k - p_l), \quad \forall l \in S \setminus \{k\}. \notag
\end{align}
If the optimal value of~\eqref{equation:feasibility} is non-positive, then $(q_k, p_k)$ is not a vertex of $\enewton$, and removing $(q_k, p_k)$ from $S$ leaves $u_S$ unchanged on $D$. Thus, such points can be safely excluded without altering the approximation on $D$.
\begin{algorithm}[!ht]
\caption{$k$-center pruning on $D$}
\label{algorithm:dqh}
\begin{algorithmic}[1]
\State \textbf{input:} $u_N$, a polyhedral function as in~\eqref{e-input};
$n$, the budget;
\State \textbf{output:} A subset $S \subseteq [N]$ with $|S|=n$ and a polyhedral approximation $u_S$;
\State \textbf{initialization:} $A = [N]$, and $S = \emptyset$;
\For{$k = 1$ to $N$}\Comment{preprocessing}
\If{$\val(F_{k, A}^D) < 0$}
\State $A = A \setminus \{k\}$;
\EndIf
\EndFor
\State choose an arbitrary index $r$ from $A$;\Comment{greedy selection}
\For{$t = 2$ to $n$} 
\State $\text{Find } r = \argmax_{k \in A} \min_{l \in S} \|(q_k, p_k) - (q_l, p_l)\|$;
\State $S = S \cup \{r\}$;
\EndFor
\State \textbf{return} $S, u_S$;
\end{algorithmic} 
\end{algorithm}

\section{Applications} \label{section:applications}
To assess our method, we apply it to a challenging example from quantum
control, following an original max-plus based approach introduced by Sridharan et al.~\cite{Sridharan_2010,Sridharan2014}. We also apply it to a pricing problem under incentive compatibility constraints, similarly to what was done by~\cite{jacquet-ackooij-2023}, using the Rochet-Chon\'e model.

\subsection{Max-plus approximation method for quantum gate synthesis} 
We first recall the max-plus approach of~\cite{Sridharan_2010,Sridharan2014} to quantum gate synthesis, and then present our numerical results.
Quantum gate synthesis is a process that involves creating efficient quantum circuits to perform specific quantum computations. An algorithm operating on a $n$-qubit system can be represented using elements from the special unitary group $SU(2^n)$. Consequently, the problem's dimension grows exponentially as $(4^n - 1)$. The system dynamics are given by:
\begin{align*}
\frac{dU}{dt} = -i \, \{\sum_{k=1}^M v_k(t) \, H_k\} \, U,\qquad
U \in SU(2^n), 
\end{align*}
with control $v \in \mathcal{V} := \{ v : \R \mapsto \R^M \mid v \text{ is} \allowbreak \text{ piecewise} \allowbreak \text{continuous}, \|v(\cdot)\| = 1\}$, initial condition $U(0)=U_0$, and such that the set of available, one and two, qubit Hamiltonians $-i H_1,\,-i H_2\,\ldots\,, -i H_ M $ are generators of the Lie algebra of $\mathrm{SU}(2^n)$. The control problem involves reaching the identity element starting from $U_0$, minimizing the cost function:
\begin{equation*}
C_t(U_0) = \inf\limits_{v \in {\mathcal{V}}} \Big \{ { \int\limits_t^{T} {\sqrt{{v(s)}^T \, R \, v(s)}} ds} + \frac{1}{\epsilon} \, \phi(U_t(v, U_0)) \Big \}.
\end{equation*}
Here, $U_t(v, U_0)$ represents the state reached from condition $U_0$ under the control $v$ applied in the time interval $[t,T]$, and the terminal cost is a penalization on the final state $\frac{1}{\epsilon} \, \phi(U) := \frac{1}{\epsilon} \langle (U - I), (U - I) \rangle$. As $\epsilon \to 0$, the penalty tends to infinity unless the state $U_t(v, U_0)$ exactly matches the identity. The scalar product is defined as $\langle U, V \rangle = \text{Re}\left( \text{tr}(U^\dagger V) \right)$, where $U^\dagger$ is the transposed complex conjugate of $U$. 

As proposed in\cite{Sridharan_2010}, we tackle the control problem by discretizing the time interval into $N$ components, each with a step size of $\tau = T / N$, and apply backpropagation as well as pruning. In this framework, we consider constant unitary controls over the intervals $[k \, \tau, (k+1) \, \tau)$. On each interval, the control is chosen from the set $\mathcal{U} = \{ v \mid \|v\| = 1, v \in \{0, \pm 1\}^M \} \cup \{ 0 \}$, which corresponds to the zero vector, the standard basis vectors of $\mathbb{R}^M$, and their negatives.

The value function $C_{k \, \tau}(U_0)$ can be expressed as:
\begin{equation*}
C_{k \, \tau}(U_0) = \inf\limits_{v \in \mathcal{U}} \sqrt{{v}^T \, R \, v} \, \tau + C_{(k+1) \, \tau}(\Phi(v) \, U_0),
\end{equation*}
where $\Phi(v) = \exp(-i \sum_{k=1}^M v_k \, H_k \, \tau)$ represents the propagator of the dynamics on the set $\mathcal{U}$ on an interval of size $\tau$. In Section III.B of~\cite{Sridharan_2010}, it is demonstrated that, with a simple reindexing, this expression can be rewritten as:
\begin{equation}
C_{k \, \tau}(U_0) = \inf\limits_{\ell \in \Lambda_k} c_\ell + P_\ell \cdot U_0,\label{e-rep}
\end{equation}
where the initial parameters are given by $c_0 = \frac{1}{\epsilon} \cdot 2 \, \text{tr}(I) = \frac{1}{\epsilon} \cdot 2^{n+1}$ and $P_0 = -\frac{2}{\epsilon} \, I$. Moreover, at the propogation step from instant $(k-1) \, \tau$ to $k \, \tau$, we rely on the Cartesian product with $\mathcal{U}$. More precisely, we have
$\Lambda_k = \mathcal{U} \times \Lambda_{k-1}$, and for any $\ell = (l, \ell^\prime) \in \Lambda_k$:
\begin{equation*}
c_\ell = c_{\ell^\prime} + \sqrt{(v^l)^T R v^l} \, \tau, \quad P_\ell = \Phi(v^l)^\dagger P_{\ell^\prime}.
\end{equation*}
The cardinality of $\Lambda_k$ is $|\Lambda_0||\mathcal{U}|^{k}$, hence it grows exponentially as the number of iterations $k$ increases. Consequently, the size of the representation~\eqref{e-rep} would also grow exponentially if no further simplifications were made. In the above cited papers, the authors tackle this difficulty by means of a {\em pruning} operator $\mathcal{P}$. i.e., at each step $k$, they replace $\Lambda_k$ by a subset $\mathcal{P}(\Lambda_k)$ of smaller cardinality, obtained by keeping only the $m$ most useful affine terms in the representation~\eqref{e-rep}. Here, $m$ is a prescribed ``budget'' (limit on the complexity of representations). In this way, the induction $\Lambda_k = \mathcal{U} \times \Lambda_{k-1}$ is replaced by $\Lambda_k =\mathcal{P}(\mathcal{U} \times \Lambda_{k-1})$, so that $|\Lambda_k|\leq m$ at each step, preventing the exponential blow up. The authors of~\cite{Sridharan_2010} formalized the notion of ``useful affine terms'' by a notion of {\em importance metric}, keeping only the $m$ affine forms with the largest importance metric. They also considered variants of this method, in which the budget is allowed to depend on the iteration. It has been noted that the pruning algorithm is a critical element in the implementation of this algorithm, allowing one to attenuate the curse of dimensionality. In particular, the benchmarks of~\cite{qu-gaubert-2011} indicate that less than $2\%$ of the execution time is used in the propagation stage, the remaining $98\%$ being used to compute the importance metric and to perform the pruning. Here, we shall apply the new pruning algorithm (Algorithm~\ref{algorithm:dqh}).
\if{
\begin{figure}[h!]
\centering
\begin{tikzpicture}[scale = 0.6, every node/.style={draw, rectangle, minimum size=0.3cm}]
\node (a0) at (0,0) {};
\node[draw=none] (s0) at (0, 3) {Step 0};

\node (b1) at (2,1.25) {};
\node (b2) at (2,0.75) {};
\node (b3) at (2,0.25) {};
\node (b4) at (2,-0.25) {};
\node (b5) at (2,-0.75) {};
\node (b6) at (2,-1.25) {};
\node[draw=none] (s1) at (2,3) {Step 1};

\foreach \x in {b1,b2,b3,b4,b5,b6} {
\draw (a0) -- (\x);
}

\node (b1p) at (4,1.25) {};
\node (b4p) at (4,-0.25) {};
\node (b5p) at (4,-0.75) {};
\node[draw=none] (s1p) at (4,3) {Step 1'};

\draw[thick, rounded corners] (1.5, -1.6) rectangle (2.5, 1.6);
\draw[->, thick] (2.8, 0) -- (3.2, 0) node[draw=none, midway, above] {$\mathcal{P}$};
\draw[thick, rounded corners] (3.5, -1.6) rectangle (4.5, 1.6);

\node (c11) at (6,2.25) {};
\node[draw=none] (c12) at (6,1.75) {$\vdots$};
\node (c16) at (6,1.25) {};
\node (c41) at (6,0.75) {};
\node[draw=none] (c42) at (6,0.25) {$\vdots$};
\node (c46) at (6,-0.25) {};
\node (c51) at (6,-0.75) {};
\node[draw=none] (c52) at (6,-1.25) {$\vdots$};
\node (c56) at (6,-1.75) {};
\node[draw=none] (s2) at (6,3) {Step 2};

\node (c11p) at (8,2.25) {};
\node (c41p) at (8,0.25) {};
\node (c51p) at (8,-1.25) {};
\node[draw=none] (s2p) at (8,3) {Step 2'};

\draw (b1p) -- (c11);
\draw (b1p) -- (c16);
\draw (b4p) -- (c41);
\draw (b4p) -- (c46);
\draw (b5p) -- (c51);
\draw (b5p) -- (c56);

\draw[thick, rounded corners] (5.5, -2.1) rectangle (6.5, 2.6);
\draw[->, thick] (6.8, 0) -- (7.2, 0) node[draw=none, midway, above] {$\mathcal{P}$};
\draw[thick, rounded corners] (7.5, -2.1) rectangle (8.5, 2.6);
\end{tikzpicture}
\caption{Quantum gate synthesis using the pruning operator $\mathcal{P}$ and two propagation steps.}

\label{fig:pruning}
\end{figure}
}\fi

Let us now consider the numerical example from~\cite{Sridharan2014}, with $M = 5$. The control set is derived from Hamiltonians of the form $H_k \in \{I \otimes \sigma_x, \, I \otimes \sigma_z, \, \sigma_x \otimes I, \, \sigma_z \otimes I, \, \sigma_x \otimes \sigma_x\}$, which include four single-body terms and one two-body term. These control directions are sufficient to generate the full Lie algebra $\mathfrak{su}(4)$, ensuring controllability. Here, $\sigma_x$ and $\sigma_z$ are Pauli matrices. The matrix $R$ is diagonal, with the first $M-1$ entries weighted by $1/r$, while the last entry has a weight of 1. 

Using the parameters, $\epsilon=0.05$, $\tau = 0.1$, $N = 6$, and $r = 3$, we compute the function $C_0$ using Algorithm~\ref{algorithm:dqh}, and plot it in the plane $\sigma_x \otimes \sigma_x$ vs. $\sigma_y \otimes \sigma_y$:
\begin{figure}[h!] 
\centering
\includegraphics[width=0.25\textwidth]{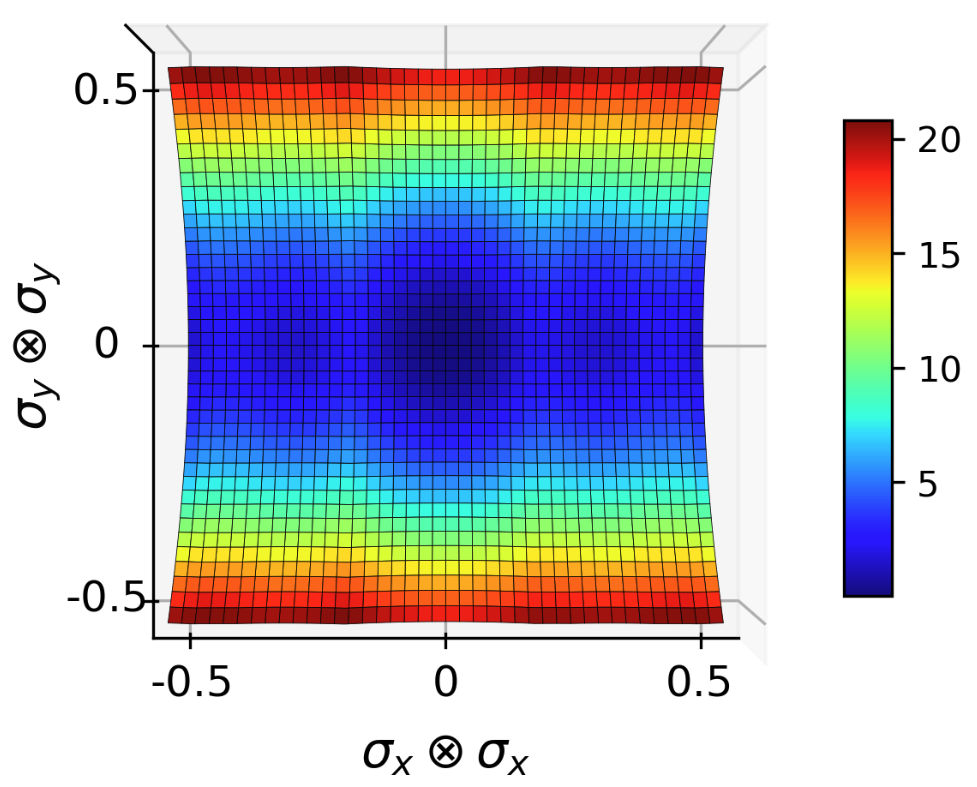}%
\caption{Plot of the value function in the plane $\sigma_x \otimes \sigma_x$ vs. $\sigma_y \otimes \sigma_y$ with the cost ratio $r$ set to 3 and a budget set to 100.} 
\label{fig:quantum:2014}
\vspace{-0.7cm} 
\end{figure}
Similarly to~\cite{Sridharan2014}, Figure~\ref{fig:quantum:2014} highlights that in the direction of the available two-body Hamiltonian $\sigma_x \otimes \sigma_x$, the cost remains relatively low compared to the direction of the two-body Hamiltonian $\sigma_y \otimes \sigma_y$, which is only accessible through Lie brackets. Additionally, small oscillations can be observed along the $y = 0$ axis. These are artifacts of the discrete propagation method caused by the finite step size of 0.1. 

Next, we set $\tau = 0.2$, $N = 50$, and $r = 1.3$ and evaluate the function using budgets ranging from 20 to 100. We then assess solution quality by computing the cost $L_1$ norm over the set: $\{\exp\left(i \, (x \, \sigma_x \otimes \sigma_x + y \, \sigma_y \otimes \sigma_y)\right) | x \in [-\pi, \pi], y \in [-\pi, \pi]\}$, under four quantization methods:
\begin{itemize}
  \item {\em $k$-center}: The greedy $k$-center algorithm~\cite{Gonzalez1985ClusteringTM} applied to $\{(-P_k, c_k)\}_{k=1}^N$.
  \item {\em $k$-center-LP}: Algorithm~\ref{algorithm:dqh} on $D = \{ X \mid Re(X_{i,j}), \allowbreak Im(X_{i,j}) \allowbreak \in [-1, 1]\}$, using LP-based preprocessing.
  \item {\em $k$-center-SDP}: Algorithm~\ref{algorithm:dqh} on the symmetric sphere $B := \{ X \text{ symmetric} \mid X X^\dagger \preceq I \}$, with SDP-based preprocessing. 
  \item {\em Primal greedy descent with SDP-based importance metric (PGD-SDP)}: The algorithm from~\cite{Gaubert_2014}, where the importance metric is computed on $B$ and the $N-n$ least important components are iteratively removed.
\end{itemize}
The experiments were conducted using an 13th Gen Intel® Core™ i5-13600H CPU @ 4.8GHz (max) with 32GB RAM, and were parallelized across 16 CPU threads using Python's {\em multiprocessing} module to accelerate computation. The LP and the SDP were solved by calling the CBC solver from the PuLP library and the MOSEK solver from the CVXPY library, respectively.
\begin{figure}[h!] 
\vspace{-0.35cm} 
\centering
\includegraphics[width=0.25\textwidth]{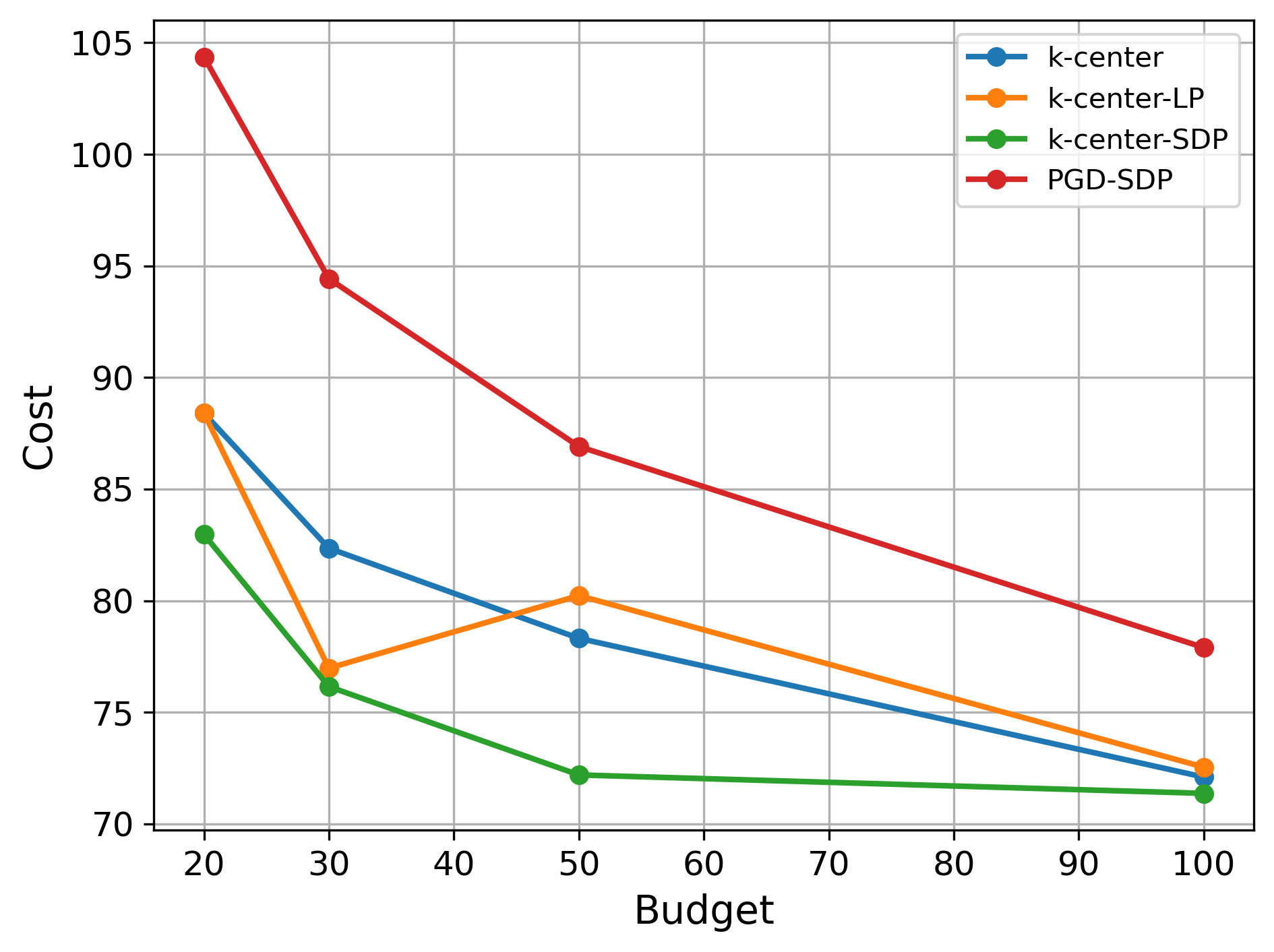}%
\vspace{-0.1cm}
\caption{Average of the value function on the plane $\sigma_x \otimes \sigma_x$ vs. $\sigma_y \otimes \sigma_y$, for different budgets, with $r$ set to 1.3.}
\label{fig:quantum:comparative}
\vspace{-0.4cm} 
\end{figure}

\begin{figure}[h!] 
\vspace{-0.35cm} 
\centering
\includegraphics[width=0.25\textwidth]{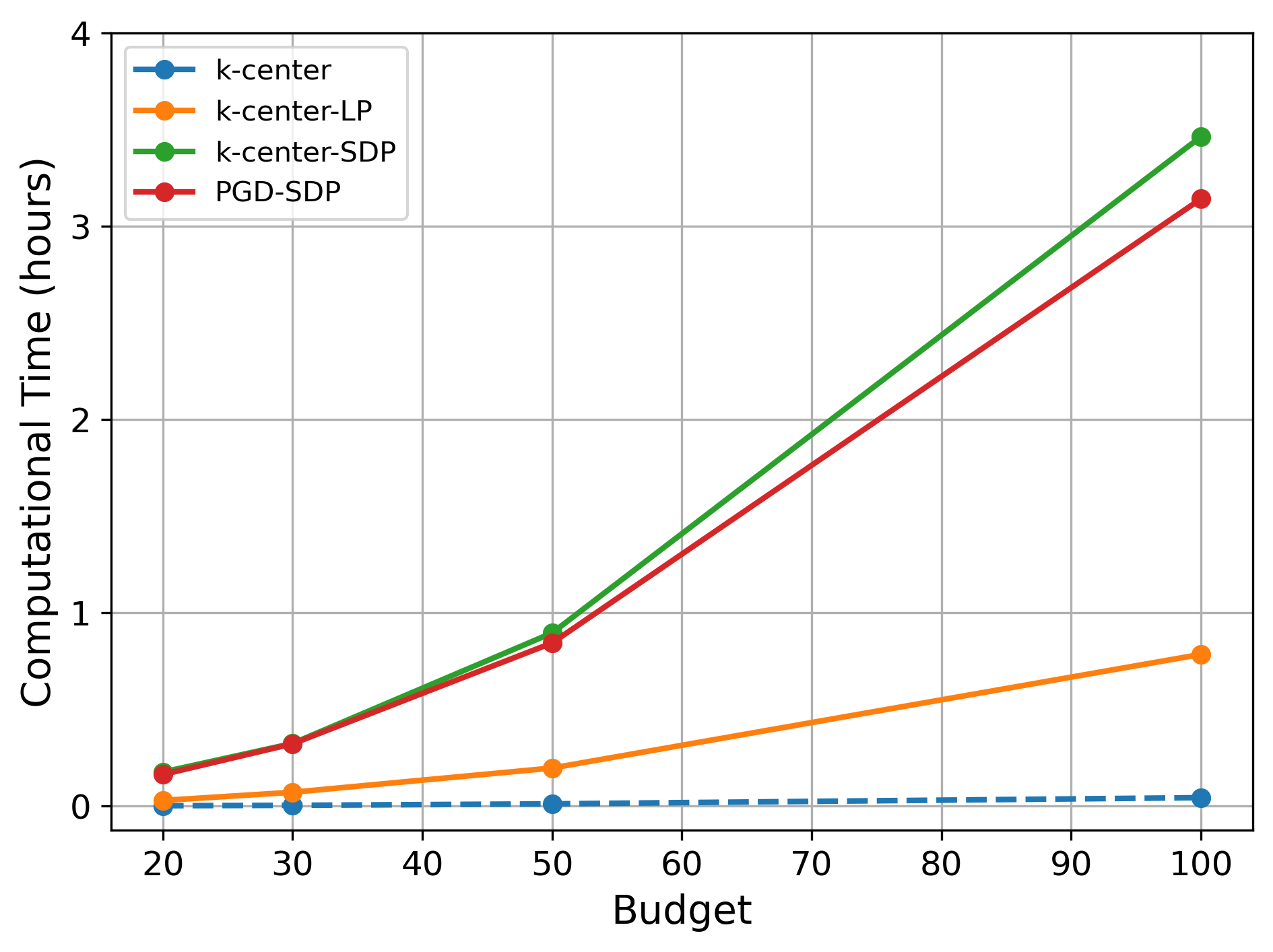}%
\vspace{-0.1cm}
\caption{Time needed to compute the value function, for different budgets, with $r$ set to 1.3.}
\label{fig:quantum:comparative:time}
\vspace{-0.35cm} 
\end{figure}

Figure~\ref{fig:quantum:comparative} highlights that, relative to the importance metric method of~\cite{qu-gaubert-2011}, the current $k$-center pruning method reduces the cost by 5-20\%, with the most pronounced improvements attained under smaller budget constraints. Among the methods considered, the SDP-based preprocessing delivers the lowest cost. Figure~\ref{fig:quantum:comparative:time} shows that the preprocessing step—whether based on LP or SDP—constitutes the main computational bottleneck relative to the $k$-center heuristic. Refining the LP preprocessing to its SDP counterpart results in a slowdown by at most a factor of 4.5, with the factor decreasing as the budget increases. This behavior arises because there is only a single SDP-type constraint, $X \in B$, whereas the number of linear constraints grows linearly with the budget. The proposed $k$-center method combined with SDP preprocessing attains a runtime comparable to that of~\cite{qu-gaubert-2011}, while achieving a lower cost (Figure~\ref{fig:quantum:comparative}). Further improvements remain possible: replacing the generic off-the-shelf SDP solver with the tailored bundle method of~\cite{Gaubert_2014} is expected to yield a comparable speedup for both methods $k$-center-SDP and PGD-SDP.

\subsection{Rochet-Chon\'e model for electricity pricing}
Pricing models often assume a hierarchical decision process between a retailer and its clients. The retailer aims to optimize an objective, such as revenue or social impact, while clients make rational decisions to maximize their utility. This structure naturally leads to a bilevel formulation, where the retailer first proposes a set of offers, and then clients respond by selecting the offer that maximizes their individual benefit. In our setting, we consider a retailer designing contracts $(q, p)$ in the context of electricity pricing. Here, $q$ represents the characteristics of the product—often interpreted as a quality vector (see, e.g.,~\cite{bergemann-yeh-zhang-2021, carlier2024general}) and $p$ denotes the corresponding price. We adopt the discretized linear-quadratic Rochet-Chon\'e model~\cite{rochet-chone-1998, ekeland-mbromberg-2009, bergemann-yeh-zhang-2021} to formulate this problem:
\begin{align} 
\max_{p, q} &\sum_{k = 1}^N \left( p_{k} - \frac{1}{2} \| {q_{k}}\|^2 \right)\rho_{{k}, N} \notag \\
\text{s.t.: } &\langle q_k, x_{k,N} \rangle - p_k \geq 0, \quad \forall k \in [N] \label{constraint:reservation},\\
& q_k \in Q, \quad \forall k \in [N] \label{constraint:available},\\
& \langle q_k, x_{k,N} \rangle - p_k \geq \langle q_l, x_{k,N} \rangle - p_l, \quad \forall k,l \in [N] \label{constraint:incentive},
\end{align}
where $N$ denotes the number of consumer segments, each grouping together individuals with similar preferences. The weight $\rho_{k, N}$ represents the proportion of consumers in segment $k$, and $x_{k,N}$ is the typical preferences of that segment. The quadratic term $\frac{1}{2}\|q_{k}\|^2$ captures the cost of providing quality $q_k$ (the quadratic norm is used for simplicity, for more realistic pricing schemes, more complex functions may be considered). The reservation constraint~\eqref{constraint:reservation} ensures that clients only accept offers that yield non-negative utility, while the availability constraint~\eqref{constraint:available} defines which offers can be proposed and which cannot. Meanwhile, the incentive compatibility constraint~\eqref{constraint:incentive} ensures that each client segment selects the most beneficial offer.

Ideally, the retailer would define a large number of segments to better represent the population and tailor offers accordingly. However, providing too many options is impractical—customers may find it difficult to compare offers, while the retailer faces challenges in management and commercialization. To address this, we apply a pruning step after solving the pricing problem, reducing the menu to a manageable number of offers. Since this reduction may leave some clients without a valid option, we explicitly include a non-participation offer that meets the reservation utility.

In our numerical experiments, we consider a linear reserve utility $R: x \mapsto \langle r, x \rangle$ to model the basic offer from the market instead of $0$. Additionally,  we use client distributions generated by EDF’s SMACH simulator~\cite{Huraux} and evaluate them in three different settings: a 2D model based on Peak and Off-Peak hours, a 3D model incorporating Peak days (white, blue, and red from \href{https://particulier.edf.fr/fr/accueil/gestion-contrat/options/tempo.html/}{EDF's Tempo contract}), and a 6D model that combines both. For each setting, we simulate a 1000 representative clients and solve the pricing problem in batches of 100 to assess average performance of the pruning in terms of revenue impact and computational effeciency. Figure~\ref{fig:pricing} presents a comparative analysis of three quantization methods: 
\begin{itemize}
\item {\em $k$-center}: The standard greedy $k$-center algorithm~\cite{Gonzalez1985ClusteringTM} applied to $\{(q_k, p_k)\}_{k=1}^N$.
\item {\em primal greedy ascent (PGA)}: The greedy algorithm from~\cite{boutsidis2015greedyminimizationweaklysupermodular} applied to the facility location model introduced in~\cite{qu-gaubert-2011}.
\item {\em primal greedy descent (PGD)}: The greedy algorithm proposed by~\cite{jacquet-ackooij-2023}.
\end{itemize}
They are evaluated under three computational budgets and in three dimensional settings: 2D, 3D, and 6D represented respectively by $\bigcirc$, $\Box$, and $\Diamond$. The experiments were conducted on the same computer without parallelization. $k$-center-LP exhibits exceptional efficiency, achieving near-optimal normalized ratios with minimal computational time, consistently across all the three dimensional settings. In comparison, PGA shows slightly lower normalized ratios and requires more computation time in 2D and 3D, with its performance further declining in 6D. Meanwhile, PGD performs better as the budget increases, which is expected since it is a descent method.

\begin{figure*}[!htp]
\vspace{0.1cm}
\centering
\begin{minipage}{0.28\textwidth} 
\centering
\includegraphics[width=0.9\textwidth]{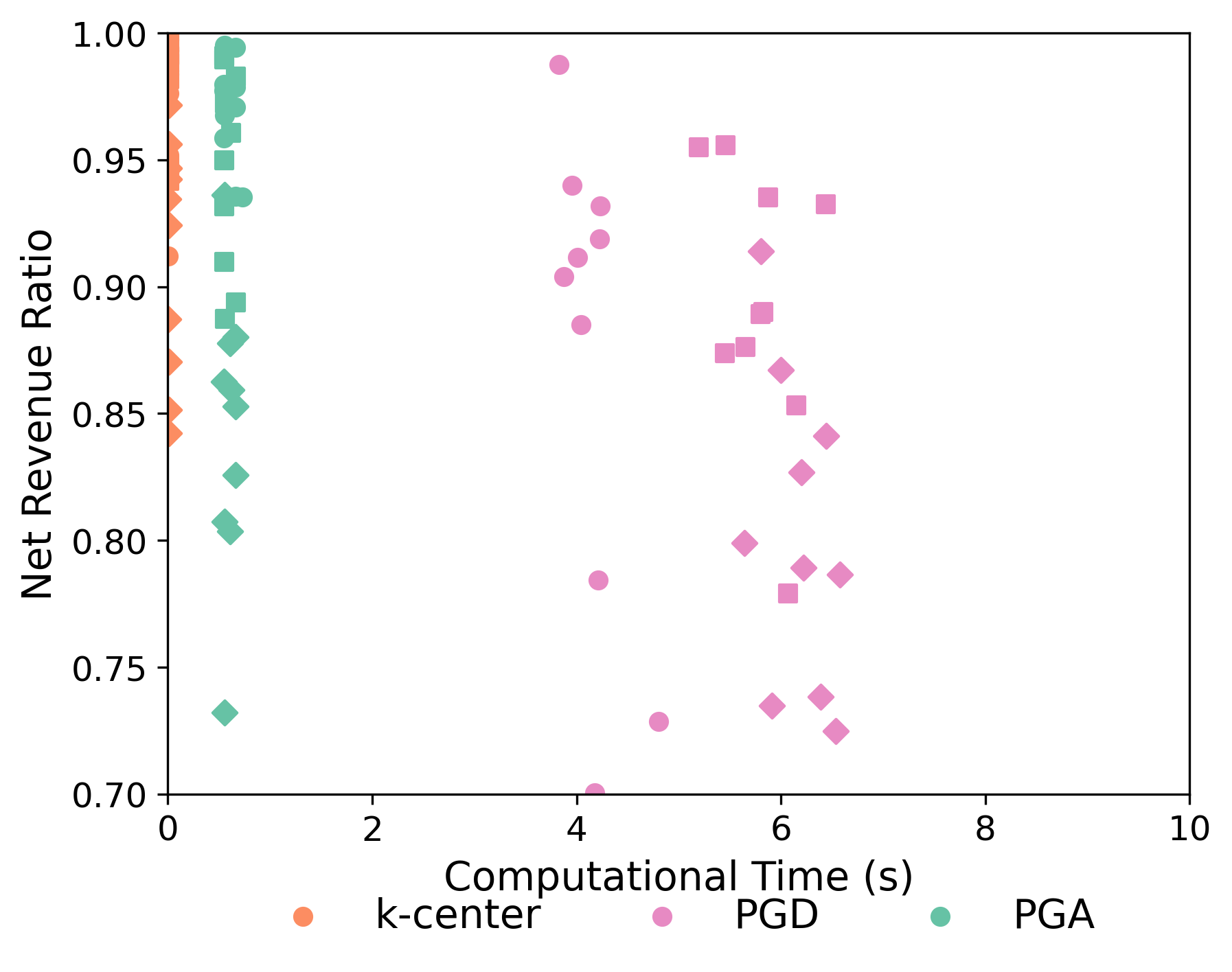}
\caption*{(a) - A budget of 10.}
\end{minipage}%
\hfill
\begin{minipage}{0.28\textwidth} 
\centering
\includegraphics[width=0.9\textwidth]{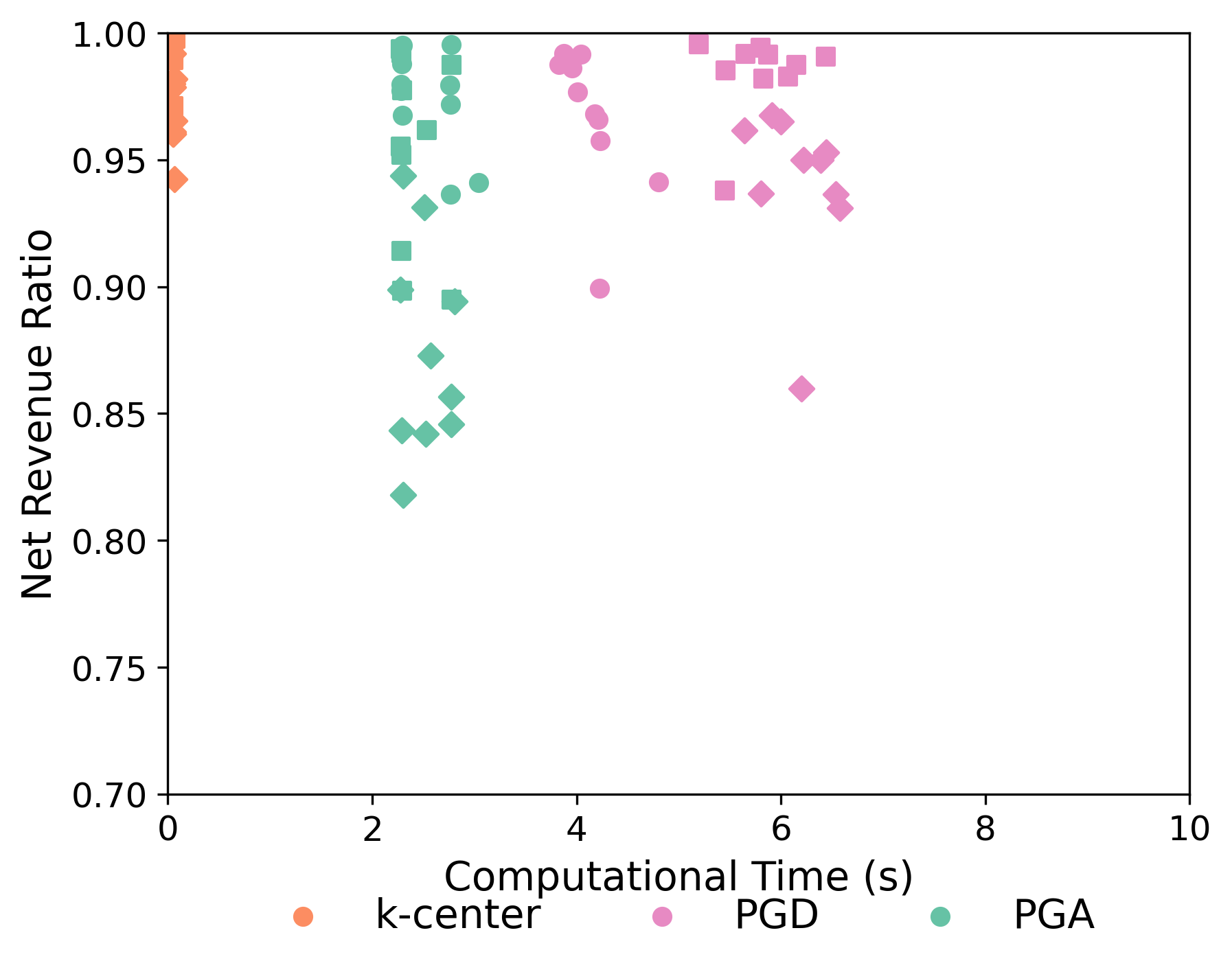}
\caption*{(b) - A budget of 25.}
\end{minipage}%
\hfill
\begin{minipage}{0.28\textwidth} 
\centering
\includegraphics[width=0.9\textwidth]{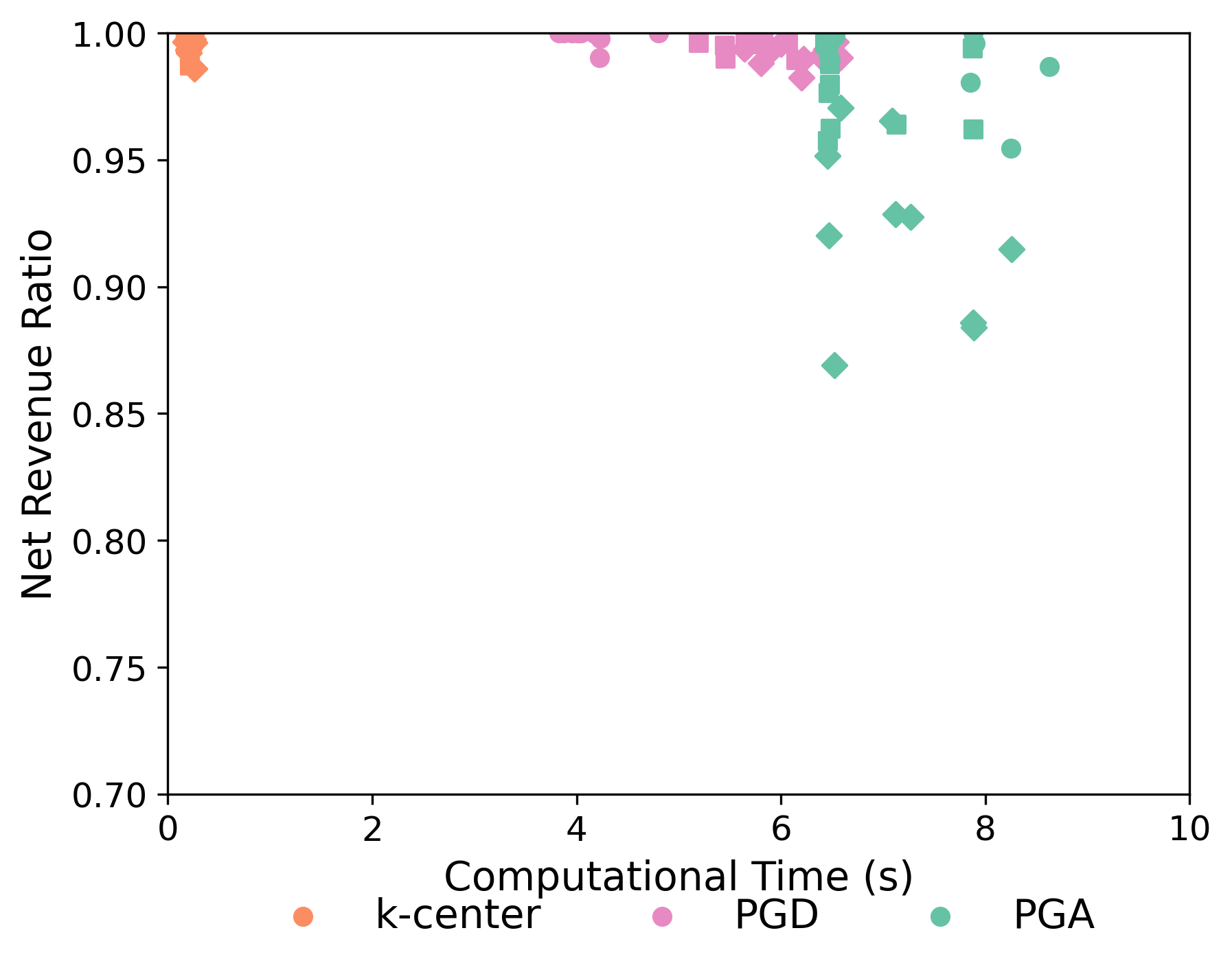}
\caption*{(c) - A budget of 50.}
\end{minipage}%
\caption{Rochet-Chon\'e problem: performance assessment of quantization methods for different budgets.}
\label{fig:pricing}
\vspace{-0.35cm} 
\end{figure*}

\section{Conclusion}
We explored the connection between polyhedral approximation and measure quantization, building on recent work by Delalande and M\'erigot. We deduced that in a certain sense, the problem of polyhedral approximation for a convex function is equivalent to the quantization problem for the Monge-Ampère measure of its Legendre–Fenchel dual.

This insight inspired a simple heuristic adapted to high dimensional instances: we prune polyhedral functions using clustering in the dual space and linear preprocessing. This has a natural geometric interpretation in terms of a lifted Newton polytope.

We validated our approach through experiments on two applications: an optimal control problem in dimension 15 (quantum gate synthesis) and a nonlinear pricing problem in electricity markets (dimension 6). Compared with earlier pruning algorithms, the present method demonstrated
improved computational efficiency and solution quality.

\bibliography{bibliography}
\bibliographystyle{ieeetr}
\end{document}